\documentclass[12pt]{article}

\usepackage[paper=letterpaper,margin=1in,twoside=false,includehead,headheight=18pt]{geometry}

\usepackage[utf8]{inputenc}

\usepackage{amsmath,amssymb,amsthm} 
\usepackage{enumerate,paralist}
\usepackage{graphicx}
\usepackage{mathtools}
\usepackage{tikz}
\usetikzlibrary{calc}

\tikzstyle{vx}=[inner sep=1pt,circle,fill=black,draw=black]

\newtheorem{theorem}{Theorem}[section]
\newtheorem{lemma}[theorem]{Lemma} 
 
\newtheorem{corollary}[theorem]{Corollary}

\newtheorem*{main}{Main Theorem}
\newtheorem*{mainconj}{Main Conjecture}

{\theoremstyle{definition} 

\newtheorem*{definition}{Definition}

 }

\newtheoremstyle{named}% 
  {}{}						% Spacing above and below; this means the default
  {\upshape}				% body font 
  {0pt}{\bfseries}			% the first arg on this line is the indent amount, second is heading font
  {.}						% Punctuation after the theorem name
  {.5em}					% Amount of space after the heading
  {\thmname{#1}\thmnote{ #3}}  % What to write for the heading

\theoremstyle{named}

\newcommand{\wo}{\setminus}

\newcommand{\set}[1]{\left\{{#1}\right\}} 
\newcommand{\setof}[2]{\left\{{#1}\,:\,{#2}\right\}}
\newcommand{\of}{\subseteq}

\newcommand{\intersection}{\bigcap}

\newcommand{\symd}{\bigtriangleup}
\newcommand{\adj}{\sim}
\newcommand{\nadj}{\nsim}

\newcommand{\s}{\sigma}

\newcommand{\N}{\mathbb{N}}

\newcommand{\ax}{A_{x\bar{y}}}
\newcommand{\axy}{A_{xy}}
\newcommand{\ay}{A_{\bar{x}y}}
\newcommand{\anxy}{A_{\overline{xy}}}
\newcommand{\G}{\mathcal{G}}
\newcommand{\Gxy}{G_{x\rightarrow y}}

\renewcommand{\complement}{\overline}
\renewcommand{\d}{\delta}
\newcommand{\D}{\Delta}

\newcommand{\ind}{i}

\renewcommand{\k}{k}

\newcommand{\w}{w}
\newcommand{\wbar}{\overline{w}}

\newcommand{\Gp}[1]{G_{#1}}

\newcommand{\tri}{%
    \tikz[scale=0.25]{%
        \foreach \pos/\name in {{(0,0)}/A,{(60:1)}/B,{(1,0)}/C}
            \path \pos node[circle, fill=black, inner sep=0.2mm] (\name) {};
        \draw[thick] (A) -- (B) -- (C) -- (A);
    }%
}

\newcommand{\cher}{%
    \tikz[scale=0.25]{%
        \foreach \pos/\name in {{(0,0)}/A,{(60:1)}/B,{(1,0)}/C}
            \path \pos node[circle, fill=black, inner sep=0.2mm] (\name) {};
        \draw[thick] (A) -- (B) -- (C);
    }%
}

\newcommand{\angel}{%
    \tikz[scale=0.25]{%
        \foreach \pos/\name in {{(0,0)}/A,{(60:1)}/B,{(1,0)}/C}
            \path \pos node[circle, fill=black, inner sep=0.2mm] (\name) {};
        \draw[thick] (A) -- (B);
    }%
}

\newcommand{\emp}{%
    \tikz[scale=0.25]{%
        \foreach \pos/\name in {{(0,0)}/A,{(60:1)}/B,{(1,0)}/C}
            \path \pos node[circle, fill=black, inner sep=0.2mm] (\name) {};
    }%
}

\DeclarePairedDelimiter{\abs}{\lvert}{\rvert}
\DeclarePairedDelimiter{\absz}{\lvert}{\rvert_0}
\DeclarePairedDelimiter{\abso}{\lvert}{\rvert_1}

\begin{document}

\title{The maximum number of complete subgraphs of fixed size in a graph with given maximum degree}
\date{\today}
\author{Jonathan Cutler\\
\small{Montclair State University}\\
\small{\texttt{jonathan.cutler@montclair.edu}}
 \and 
A.J.~Radcliffe\\
\small{University of Nebraska-Lincoln}\\
\small{\texttt{jamie.radcliffe@unl.edu}}}
\maketitle
\begin{abstract}
	In this paper, we make progress on a question related to one of Galvin that has attracted substantial attention recently.  The question is that of determining among all graphs $G$ with $n$ vertices and $\D(G)\leq r$, which has the most complete subgraphs of size $t$, for $t\geq 3$.  The conjectured extremal graph is $aK_{r+1}\cup K_b$, where $n=a(r+1)+b$ with $0\leq b\leq r$.  Gan, Loh, and Sudakov proved the conjecture when $a\leq 1$, and also reduced the general conjecture to the case $t=3$.  We prove the conjecture for $r\leq 6$ and also establish a weaker form of the conjecture for all $r$.
\end{abstract}

\maketitle
%\tableofcontents

\section{Introduction} % (fold)
\label{sec:intro} 

Galvin \cite{G} conjectured that the graph on $n$ vertices with minimum degree at least $d$ having the most independent sets is $K_{d,n-d}$, provided $n\geq 2d$.  He proved the conjecture for $d=1$, and also for fixed $d$ and large $n$.  Alexander, Mink, and the first author \cite{ACM} proved this conjecture for bipartite graphs.  Further, Galvin conjectured that the same graph has the largest number of independent sets of each size $t\geq 3$, even though it has fewer independent sets of size two than a $d$-regular graph.  Engbers and Galvin \cite{EG} proved this stronger conjecture for $d=2,3$.  Other progress on the stronger conjecture was made by Law and McDiarmid \cite{LM} and Alexander and Mink \cite{AM}. Quite recently, Gan, Loh, and Sudakov \cite{GLS} proved this conjecture. We write $\ind_t(G)$ for the number of independent sets in $G$ of size $t$.

\begin{theorem}[Gan, Loh, Sudakov \cite{GLS}]\label{thm:gls}
    If $G$ is a graph with $n$ vertices having $\d(G)\ge d$ where $n\ge 2d$ then 
    \[
        \ind_t(G) \le \ind_t(K_{d,n-d}),
    \]
    provided $t\ge 3$.
\end{theorem}

This left unanswered the question of what is the optimal graph when $n<2d$. In this range it is cleaner to state the problem in its complementary version: counting complete subgraphs in a graph of given maximum degree. In fact, this is the version of the problem on which Gan, Loh, and Sudakov worked. The current authors \cite{CR14} answered this question for the total number of complete subgraphs. For a graph $G$, we let $\k(G)$ be the number of complete subgraphs (which we also call cliques) in $G$ and $\k_t(G)$ be the number of cliques of size $t$ in $G$.

\begin{theorem}[Cutler, Radcliffe \cite{CR14}]\label{thm:ours}
	For all $n,r\in \N$, write $n=a(r+1)+b$ with $0\le b\le r$. If $G$ is a graph on $n$ vertices with $\Delta(G)\leq r$, then 
	\[
		\k(G) \leq \k(aK_{r+1} \cup K_b),
	\]
	with equality if and only if $G=aK_{r+1}\cup K_b$, or $r=2$ and $G=(a-1)K_3\cup C_4$ or $(a-1)K_3\cup C_5$.
\end{theorem}

The case $a=1$ of this theorem corresponds (in the complement) to Galvin's original conjecture (i.e., $n\geq 2d$).
The case $b=0$ was proved independently by Engbers and Galvin \cite{EG} and Wood \cite{W}.  What now remains is the level set version to Theorem~\ref{thm:ours}, conjectured by Gan, Loh, and Sudakov.

\begin{mainconj}[Gan, Loh, Sudakov \cite{GLS}]
	For all $n,r\in \N$, write $n=a(r+1)+b$ with $0\le b\le r$. If $G$ is a graph on $n$ vertices with $\Delta(G)\leq r$ and $t\geq 3$ is an integer, then 
	\[
		\k_t(G) \leq \k_t(aK_{r+1} \cup K_b).
	\]
\end{mainconj}

One striking consequence of the approach of Gan, Loh, and Sudakov in \cite{GLS} is a proof that the case $t=3$, i.e., the case of triangles, implies the general result.

\begin{theorem}[Gan, Loh, Sudakov \cite{GLS}]
	If the Gan-Loh-Sudakov Conjecture is true for $t=3$, then it is true for $t>3$.
\end{theorem}

In this paper we make partial progress on the Gan-Loh-Sudakov Conjecture. 

\begin{main}
	Let $G$ be a graph with $n$ vertices and $\D(G)\leq r$.  Write $n=a(r+1)+b$ where $0\leq b\leq r$.
	\begin{enumerate}
		\item There exists $a'\ge 0$ and some graph $H$ with $n(H)\le 2r(r+1)/(3\sqrt{3})$ such that 
    \[
        n = a'(r+1) + n(H) \qquad\text{and}\qquad \k_t(G) \leq \k_t(a'K_{r+1} \cup H).
    \]
		\item If $r\leq 6$, then $k_3(G)\leq k_3(aK_{r+1}\cup K_b)$, i.e., the Gan-Loh-Sudakov Conjecture is true for $r\leq 6$.
	\end{enumerate}   
\end{main}

Following a similar approach to that in \cite{CR14}, we analyze structures in $G$ that we call \emph{clusters}, maximal cliques with maximum-size common neighborhoods.  We aim to prove that clusters are either \emph{foldable} or \emph{dischargeable}.  In Section~\ref{sec:fad}, we introduce our techniques.  In Section~\ref{sub:folding}, we discuss a folding operation which performs a local modification, turning a cluster and its neighborhood into a $K_{r+1}$.  A cluster is foldable if this increases the number of triangles.  In Section~\ref{sub:mu}, we discuss a parameter $\mu$ associated with a cluster which determines its foldability.  In Section~\ref{sub:discharging}, we explain a discharging technique for proving that the average number of common neighbors of adjacent vertices is small.  Finally, in Section~\ref{sec:theorem}, we prove the Main Theorem.

Our notation is standard.  We let $[n]=\set{1,2,\ldots,n}$ and, for a set $S$, we let $\binom{S}k=\setof{X\of S}{\abs{X}=k}$.

% section introduction (end)

\section{Folding and discharging} % (fold)
\label{sec:fad}

\subsection{Folding} % (fold)
\label{sub:folding}

Our approach to the proof of the Main Theorem is as follows. We consider a graph $G$ on $n$ vertices with $\D(G)\le r$.  We assign weights to the edges of $G$: if $xy\in E(G)$ then we set
\[
	\w(xy) = \abs*{N(x)\cap N(y)},
\]
the number of common neighbors of $x$ and $y$.  Equivalently, $\w(xy)$ is the number of triangles of $G$ containing the edge $xy$.  Clearly,
\begin{equation*}\label{eq:kbound}
	\k_{3}(G) = \frac1{3} \sum_{xy\in E(G)} \w(xy).
\end{equation*}
Thus if the average weight of edges is small then there will not be many triangles in total. Note that, since $\D(G)\leq r$, we have $\w(xy) \le r-1$ for all $xy\in E(G)$.
\begin{definition}
	An edge $xy\in E(G)$ is called \emph{tight} if $\w(xy)=r-1$.  A clique in $G$, all of whose edges are tight, is called a \emph{tight clique}.  A maximal tight clique is a \emph{cluster}.
\end{definition}
Note that the edge $xy$ is tight exactly if $N[x]=N[y]$ (where $N[x]$ is the closed neighborhood of $x$) and $d(x)=d(y)=r$.  Hence, the tight edge graph is the disjoint union of clusters of $G$.

If $T$ is a cluster, we will want to think about its set of common neighbors, $S=\bigcap_{x\in T} N(x)$.  Each vertex of $T$ has closed neighborhood $T\cup S$.  Vertices in $S$ can have neighbors in $T$, $S$, and the rest of the graph.  All vertices in $T$ have degree $r$, but this is not necessarily true for vertices in $S$.

\begin{definition}
	Suppose that $G$ is a graph with $\D(G)\le r$ and $T\of V(G)$ is a cluster of size $t$. We let $S_T = \intersection_{x\in T} N(x)$ and define a new graph by converting $T\cup S_T$ into a clique (of size $r+1$) and deleting all the edges $[S_T,V(G)\wo(T\cup S_T)]$. In other words we define the \emph{folding of $G$ at $T$} by
	\[
		\Gp{T} = G + \binom{S_T}{2} - [S_T,V(G)\wo(T\cup S_T)],
	\]
where, for sets $U$ and $V$, $[U,V]=\setof{uv}{u\in U, v\in V}$.  Also, we define
	\[
		R_T = \complement{G[S_T]},
	\]
i.e., the graph on $S_T$ whose edges are those not in $G$.  Note that since $T$ is a maximal tight clique, $\d(R)\geq 1$.
\end{definition}

Throughout the remainder of the paper, if $T$ is a cluster, we consistently write $t$ for $\abs{T}$, $S$ for $S_T$, $s$ for $\abs{S_T}$, and $R$ for $R_T$, when there is no potential for confusion.

\begin{lemma}\label{lem:dR}
	With the notation of the previous definition, for all $x\in S$ we have
	\[
		\abs{N_G(x)\wo (T\cup S)} \le d_{R}(x),
	\]
	and, in particular, $\D(G_T)\leq r$.
\end{lemma}

\begin{proof}
	The first inequality is immediate and the second follows from the fact that, after folding, vertices in $T\cup S$ have degree exactly $r$ and no other vertex has its degree increased. 
\end{proof}

We will show that in a wide range of circumstances, we have $k_3(\Gp{T})\geq k_3(G)$.  In the rest of this section, we bound the difference $k_3(\Gp{T})-k_3(G)$.  We will consider the induced subgraphs of $R$ having three vertices, and their configurations.

\begin{definition}
	If $H$ is a graph on three vertices, then we let $\#(G,H)$ be the number of $3$-subsets of $V(G)$ inducing a graph isomorphic to $H$.
\end{definition} 

\begin{lemma} 
	If $T$ is a cluster, then
	\[
		k_3(\Gp{T})-k_3(G)\geq t\cdot e(R)-2\#(R,\tri)-\#(R,\cher).
	\]
\end{lemma}

\begin{proof}
	We start by proving that 
	\begin{align}
		k_3(\Gp{T})-k_3(G)&\geq t\cdot e(R)+\#(R,\tri)+\#(R,\cher)+\#(R,\angel)\notag\\ 
			  & \qquad \qquad \qquad - \sum_{v\in S} \binom{d(v)}2-\sum_{vw\not\in E(R)} \min(d(v),d(w)),\label{eqn:first}
	\end{align}
	where we temporarily write $d(v)$ for $d_R(v)$.  To prove (\ref{eqn:first}), we will count exactly the new triangles in $\Gp{T}$ and bound the number of triangles lost from $G$.  The triangles gained are exactly triples in $S\cup T$ containing an edge of $R$.  This number is $t\cdot e(R)+\#(R,\tri)+\#(R,\cher)+\#(R,\angel)$.  Triangles lost from $G$ contain either one vertex of $S$ or two vertices of $S$ not adjacent in $R$.  There are at most $\sum_v \binom{d(v)}2$ of the first type since such a triangle contains two $G$-neighbors of $v$ outside $S\cup T$ that are adjacent in $G$.  Triangles of the second type contain two vertices $v,w\in S$ which are adjacent in $G$ (and so nonadjacent in $R$) and a common $G$-neighbor of $v$ and $w$ outside $S\cup T$, of which there are at most $\min(d(v),d(w))$.

	We will now bound the summation terms by counting configurations of triples.  Note first that
	\begin{equation}
		\sum_{v\in S} \binom{d(v)}2 = \#(R,\cher)+3\#(R,\tri).\label{eqn:second}
	\end{equation}
	The left hand side counts the number of pairs of incident edges of $R$.  Each such pair of edges determines a unique triple of vertices.  Triples of the form $\tri$ contribute three incident pairs; triples of the form $\cher$ contribute one.  The other summation term we bound as follows.
	\begin{align}
		\sum_{vw\not\in E(R)} \min(d(v),d(w)) &\leq \frac{1}2\sum_{vw\not\in E(R)} (d(v)+d(w))\notag\\
		&=\#(R,\cher)+\#(R,\angel).\label{eqn:third}
	\end{align}
	The first bound is clear.  The second equality follows from the fact that the sum counts edge/non-edge incidences, and these incidences each determine a unique triple.  Triangles and empty triples have no such incidences.  Two-edge and one-edge triples each have two edge/non-edge incidences.  Combining (\ref{eqn:first}), (\ref{eqn:second}), and (\ref{eqn:third}), the lemma is proved.
\end{proof}

\begin{definition}
	Given a graph $G$, we let $\mu(G)=2\#(G,\tri)+\#(G,\cher)$.  We call triples inducing $\tri$, of course, \emph{triangles}, and those inducing $\cher$, \emph{cherries}.
\end{definition}

\begin{corollary}
	The net gain of triangles when folding at $T$ is at least $t\cdot e(R)-\mu(R)$.  Hence, if $t\cdot e(R)\geq \mu(R)$, then $k_3(\Gp{T})\geq k_3(G)$.
\end{corollary}

\begin{definition}
	We say that $T$ is \emph{foldable} if $t\cdot e(R)\geq \mu(R)$.
\end{definition}

% subsection intro (end) 

\subsection{Extremal results for $\mu(G)$} % (fold)
\label{sub:mu}

In this section, we answer two natural extremal questions about the parameter $\mu$ which will be useful in the proof of the Main Theorem.  Firstly, we prove that given $n$ and $m$, the maximum value of $\mu$ for graphs with $n$ vertices and $m$ edges is attained by the lex graph.  Secondly, we prove that the maximum value of $\mu$ among graphs with $n$ vertices, $m<n-1$ edges, and minimum degree at least one is attained by the union of a star of maximum size and a collection of disjoint edges.

\begin{definition}
	The \emph{lex graph} on $n$ vertices and $m$ edges, denoted $L(n,m)$, is the graph with vertex set $[n]$ and edge set consisting of the initial segment of size $m$ according to the lex order on $\binom{[n]}2$.  The \emph{lex order} on $\mathcal{P}([n])$ is defined by $A<B$ if $\min(A\symd B)\in A$.  Thus, the first few edges in lex order are
	\[
		\set{1,2},\set{1,3},\set{1,4},\ldots,\set{1,n},\set{2,3},\set{2,4},\ldots,\set{2,n},\set{3,4},\ldots.
	\]
\end{definition}

We will adopt a standard approach to show that the lex graph is extremal.  We first prove that $\mu$ is maximized by some threshold graph, and then show that among threshold graphs, lex is best.  There are many equivalent definitions of threshold graphs (see, e.g., \cite{MP}) but in this paper, we find the following most convenient.

\begin{definition}
	A \emph{threshold graph} is a graph that is constructed inductively from $K_1$ by successively adding either an isolated vertex or a dominating vertex.  We think of the vertex set of a threshold graph on $n$ vertices as $[n]$.  We label the vertices ``backwards'', so that vertex $n$ is the initial $K_1$, vertex $n-1$ is the first vertex added, and so on.  The \emph{code} of a threshold graph with $n$ vertices is the binary sequence $\s$ of length $n-1$ such that $\s_i=1$ if the vertex $i$ was a dominating vertex when added and $\s_i=0$ if it was an isolate.  For a binary sequence $\s\in \set{0,1}^{n-1}$, we let $T(\s)$ denote the threshold graph on $[n]$ with code $\s$.  Note that (vertices that correspond to) $1$s are adjacent everything to their right and $1$s to their left, while $0$s are adjacent only to $1$s to their left.
	
	We will write an expression such as $1^a0^b1^c\ldots$ for the sequence consisting of $a$ $1$s, $b$ $0$s, $c$ $1$s, etc.  Also, given sequences $\s$ and $\tau$, we write $\s.\tau$ for their concatenation.  Thus, for instance, $T(1^40^5)$ is the join\footnote{We define the \emph{join} of graphs $G$ and $H$ with disjoint vertex sets, denoted $G\vee H$, to be the graph with vertex set $V(G)\cup V(H)$ and edge set $E(G)\cup E(H)\cup \setof{xy}{x\in V(G), y\in V(H)}$.} of $K_4$ and  $E_6$, whereas $T(0^61^3)=K_4\cup E_6$.  For a binary sequence $\s$, we let $\abs{\s}$ be its length.  Also, we let $\absz{\s}$ and $\abso{\s}$ be the number of $0$s and $1$s, respectively, in $\s$.
\end{definition}

The lex graph $L(n,m)$ is a threshold graph.  Its code is of the form $1^a0^b1^x0^c$ for some $a,b,c\geq 0$ and $x\in \set{0,1}$.  As $m$ increases, $1$s march across from the right, accumulating at the left.

Our approach to proving that $\mu$ is maximized by some threshold graph is to use a compression that makes a graph ``more threshold".  We will show that this compression increases the value of $\mu$.  This implies, as in \cite{CR2}, that some threshold graph maximizes $\mu$.  Let $G$ be any graph, and let $x$ and $y$ be distinct vertices in $G$. The choice of $x$ and $y$ defines a natural partition of $V(G\setminus \set{x,y})$ into four parts: vertices that are adjacent only to $x$, vertices adjacent only to $y$, vertices adjacent to both, and vertices adjacent to neither.  We write
\begin{align*}
	\ax &=\setof{v\in V(G\setminus \set{x,y})}{v\adj x, v\nadj y},\\
	\axy&=\setof{v\in V(G\setminus \set{x,y})}{v\adj x, v\adj y},\\
	\ay &=\setof{v\in V(G\setminus \set{x,y})}{v\nadj x, v\adj y}, \text{ and}\\
	\anxy&=\setof{v\in V(G\setminus \set{x,y})}{v\nadj x, v\nadj y}.
\end{align*} 
\begin{definition} The
\emph{compression of $G$ from $x$ to $y$}, denoted
$\Gxy$, is the graph obtained from $G$ by deleting all edges between $x$ and $\ax$ and adding all
edges from $y$ to $\ax$. 
\end{definition}

The following lemma from \cite{CR2} summarizes a useful property of the compression operator.  When we use the lemma, $\G$ will be the set of graphs with $n$ vertices and $m$ edges having $\mu(G)$ maximum.

\begin{lemma}[Cutler, Radcliffe \cite{CR2}]\label{lem:max}
	Suppose that $\G$ is a family of graphs on a fixed vertex set $V$.  If $\G$ is closed under all compressions, i.e., for any $G'\in \G$ and any $x,y\in V$ we also have $\Gxy'\in \G$, then $\G$ contains a threshold graph.
\end{lemma}

We now show that compression does not decrease $\mu$.

\begin{lemma}\label{lem:mu}
	If $G$ is a graph with distinct vertices $x$ and $y$, then $\mu(G)\leq \mu(\Gxy)$.
\end{lemma}

\begin{proof}
	We will compute $\mu(\Gxy)-\mu(G)$ by counting triangles and cherries that appear in only one of the two graphs.  Such triples must consist of a pair with at least one vertex from $\ax$ together with $x$ or $y$.  Figure~\ref{fig:bigguy} shows almost all of the types of pairs we need to consider, and Table~\ref{tab:bigguy} shows the contribution of each kind of pair to both $\mu(G)$ and $\mu(\Gxy)$.
	\begin{figure}[ht]
        \newcommand{\Axy}{A_{xy}}
        \newcommand{\Ax}{A_{x\bar{y}}}
        \newcommand{\Ay}{A_{\bar{x}y}}
        \newcommand{\Anxy}{A_{\overline{xy}}}
        \begin{center}
            \begin{tikzpicture}
                \def\xx{1}  \def\xy{8}
                \def\yx{7}  \def\yy{\xy}
        
                \def\axyx{4}  \def\axyy{\xy}
                \def\axx{\xx} \def\axy{4}
                \def\ayx{\yx} \def\ayy{\axy}
                \def\anxyx{\axyx} \def\anxyy{2}
        
                \def\eone{1} \def\etwo{2}
                \def\horiz{2 and 1} \def\vertic{1 and 2}
        
                \pgfmathsetmacro{\qx}{\axyx+\eone^2/(\yx-\axyx)}
                \pgfmathsetmacro{\qy}{\axyy+\etwo*sqrt(1-\eone^2/(\yx-\axyx)^2)}
                \pgfmathsetmacro{\tx}{\ayx-\etwo*sqrt(1-\eone^2/(\yy-\ayy)^2)}
                \pgfmathsetmacro{\ty}{\ayy+\eone^2/(\yy-\ayy)}
        
                \draw (\axyx,\axyy) ellipse (1 and 2) node (axy) {$\Axy$};
                \draw (\axx,\axy) ellipse (2 and 1) node (ax) at (1,3.5) {$\Ax$};
                \draw (\ayx,\ayy) ellipse (2 and 1) node (ay) at (7,3.5) {$\Ay$};
                \draw (\anxyx,\anxyy) ellipse (2 and 1) node (anxy) {$\Anxy$};
        
                \path coordinate (x) at (\xx,\xy)
                      coordinate (y) at (\yx,\yy)
                      coordinate (P) at (2*\axyx-\qx,\qy)
                      coordinate (Q) at (\qx,\qy)
                      coordinate (R) at (2*\axyx-\qx,2*\axyy-\qy)
                      coordinate (S) at (\qx,2*\axyy-\qy)
                      coordinate (T) at (\tx,\ty)
                      coordinate (U) at (2*\ayx-\tx,\ty);
        
                % \draw(-3,-3)  grid[help lines] (10,10);
        
                \draw (x)--(P)
                      (x)--(R)
                      (y)--(Q)
                      (y)--(S)
                      (y)--(T)
                      (y)--(U);
        
                \node[vx, label=left:$x$] at (x) {};
                \node[vx, label=right:$y$] at (y) {};
        
                \path (-0.75,4.2) node[vx] (1a) {}
                           +(1,0) node[vx] (1b) {}
                      (-0.75,3.8) node[vx] (2a) {}
                           +(1,0) node[vx] (2b) {}
                        (1.7,4.8) node[vx] (3a) {}
                       +(-30:0.4) node[vx] (4a) {}
                        (3.8,6.5) node[vx] (4b) {}
                       +(150:0.4) node[vx] (3b) {}
                        (1.7,3.2) node[vx] (6a) {}
                        +(30:0.4) node[vx] (5a) {}
                        (2.6,2.5) node[vx] (6b) {}
                        +(30:0.4) node[vx] (5b) {}
                       (2.75,4.2) node[vx] (7a) {}
                         +(2.5,0) node[vx] (7b) {}
                       (2.75,3.8) node[vx] (8a) {}
                         +(2.5,0) node[vx] (8b) {};
                 
                \draw (1a)-- node[above] {$1$} (1b)
                      (3a)-- node[above left] {$3$} (3b)
                      (5a)-- node[above right] {$5$} (5b)
                      (7a)-- node[above] {$7$} (7b);
              
                \draw[dashed] (2a)-- node[below] {$2$} (2b)
                              (4a)-- node[below right] {$4$} (4b)
                              (6a)-- node[below left] {$6$} (6b)
                              (8a)-- node[below] {$8$} (8b);

                \end{tikzpicture}
        \end{center}
		\caption{Types of pairs involved in changing triples under compression.}\label{fig:bigguy}
	\end{figure}
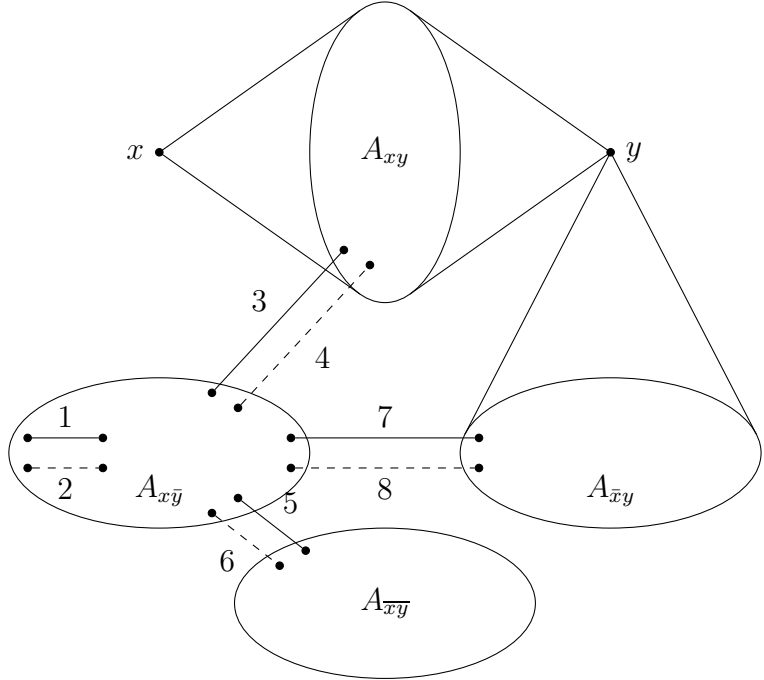
In Table~\ref{tab:bigguy}, we use $x$-$\tri$ to denote a pair that forms a triangle with $x$, and so on.
	\begin{table}
		\begin{center}
		\begin{tabular}{ccc}
			Type of edge & Configurations in $G$ & Configurations in $\Gxy$\\ \hline
			$1$ & $x$-$\tri$ & $y$-$\tri$\\ 
			$2$ & $x$-$\cher$ & $y$-$\cher$\\
			$3$ & $x$-$\tri$ & $y$-$\tri$\\
			& $y$-$\cher$ & $x$-$\cher$\\
			$4$ & $x$-$\cher$ & $y$-$\cher$\\
			$5$ & $x$-$\cher$ & $y$-$\cher$\\
			$6$ & --- & ---\\
			$7$ & $x$-$\cher$ & $y$-$\tri$\\
			& $y$-$\cher$ & \\
			$8$ & --- & $y$-$\cher$
		\end{tabular}
		\end{center}
		\caption{Contributions to $\mu$ in $G$ and $\Gxy$.}\label{tab:bigguy}
	\end{table}
	Note that edges of type $7$ contribute to two cherries in $G$ and one triangle in $\Gxy$, but these combinations make the same contribution to $\mu$.  The only type of triple not covered in Table~\ref{tab:bigguy} is one of the form $axy$ with $a\in \ax$.  If $x\nadj y$, then this triple  contributes neither to $\mu(G)$ nor to $\mu(\Gxy)$.  If $x\adj y$, the triple $axy$ is a cherry in both $G$ and $\Gxy$, though with different edges.
\end{proof}

\begin{theorem}\label{thm:lex}
	If $G$ is a graph with $n$ vertices and $m$ edges, then $\mu(G)\leq \mu(L(n,m))$.
\end{theorem}

\begin{proof}
	Let
	\[
		\G = \setof{G'}{n(G')=n, e(G')=m, \text{$\mu(G')$ is maximum}}.
	\]
	By Lemma~\ref{lem:max}, $\G$ contains a threshold graph and so we may assume that $G\in \G$ and is threshold.  Say $G=T(\s)$.  We will show that $\s$ cannot be of the form $\phi.01.\psi.10.\rho$ (even if one or all of $\phi$, $\psi$, or $\rho$ are empty) nor $\phi.01.\psi.1$.  If so, we will compare $G$ to $G'=T(\s')$ where, in the first case, $\s'=\phi.10.\psi.01.\rho$, and in the second, $\s'=\phi.10.\psi.0$.  The second case is that where the second $01$ extends past the end of the code to involve vertex $n$.  Our proof works the same in both cases with $\abs{\rho}=0$ in the second case.  It is easy to check that $e(G')=e(G)$, and we will show that $\mu(G')> \mu(G)$.  We will then show that the only threshold graphs whose codes do not contain these patterns are lex graphs.
	
	Our first job is to show $\mu(G')\geq \mu(G)$.  Many of the triples that contribute to $\mu(G)$ and $\mu(G')$ effectively stay the same between the two graphs.  In $\s$, the $01$ is in positions $a$ and $a+1$ where $a=\abs{\phi}+1$ and the $10$ is in positions $b$ and $b+1$ where $b=\abs{\phi}+\abs{\psi}+3$.  If we define the involution $f=(a\ a+1)(b\ b+1)$, the product of two transpositions, then typically a triple $Q$ is a triangle or cherry in $G$ exactly if $f(Q)$ is a triangle or cherry in $G'$.  We compute the difference $\mu(G')-\mu(G)$ by enumerating the exceptions.  In fact, we have $G'=f(G)+\set{a,a+1}-\set{b,b+1}$.  Thus all triangles or cherries gained or lost must contain either $\set{a,a+1}$ or $\set{b,b+1}$.  Considering triangles first, there are triangles in $G$ of the form $Q=\set{x,b,b+1}$ where $\s_x=1$ and $x<b$ for which $f(Q)$ is not a triangle in $G'$.  Conversely, there are triangles $Q'=\set{y,a,a+1}$ where $\s'_y=1$ and $y<a$ for which $f(Q')$ is not a triangle in $G$.  Hence, the contribution of triangles to $\mu(G')-\mu(G)$ is
	\[
		-2(\abso{\phi}+1+\abso{\psi}-\abso{\phi})=-2(1+\abso{\psi}).
	\]
	As for cherries, $G$ has cherries of the form $R=\set{b,b+1,x}$ where $x>b+1$ and also cherries of the form $R=\set{y,a,a+1}$ where $\s_y=1$ and $y<a$, for which $f(R)$ is not a cherry in $G'$.  Conversely, there are cherries in $G'$ of the form $\set{y,b,b+1}$ with $\s'_y=1$ and $y<b$ and $\set{a,a+1,x}$ with $x>a+1$ whose images under $f$ are not cherries in $G$.  Thus, the contribution to $\mu(G')-\mu(G)$ is 
	\[
		(\abs{\psi}+2+\abs{\rho})-\abs{\rho}+(\abso{\phi}+1+\abso{\psi})
		-\abso{\phi}=\abs{\psi}+\abso{\psi}+3.
	\]
	Combining the two contributions, we see that
	\[
		\mu(G')-\mu(G)=\abs{\psi}+\abso{\psi}+3-2(1+\abso{\psi})
		=\absz{\psi}+1>0.
	\]
	
	It only remains to show that if $\s$ is not of either of the forms above, then $T(\s)$ is the lex graph.  We will call a $01$ substring a \emph{rise}, and a $10$ substring a \emph{fall}. We have shown that $\s$ does not have a rise followed by a (disjoint) fall, nor does it have a rise followed by a final $1$. We wish to show that $\s$ is the code of a lex graph.  Suppose first that $\s$ has no rises; in that case $\s$ is of the form $1^a0^b$ for some non-negative $a,b$ and hence $G$ is a lex graph. Otherwise we look for the first rise, and split $\s$ there: $\s=\phi'.01.\psi'=\phi.\psi$, where $\phi=\phi'.0$ and $\psi=1.\psi'$. Since $\phi$ does not contain a rise it is of the form $1^a0^b$ with $a\ge 0, b\ge 1$. Also, by the condition on $\s$ we know that $\psi'$ does not contain a fall, nor does it end with a $1$.  This implies that $\psi=0^c$ for some $c\ge 0$. Hence $\s=1^a0^b10^c$ is the code for a lex graph.
\end{proof}

In our application to our proof of the Main Theorem, the graphs for which we want to bound $\mu$ are required to have minimum degree at least one.

\begin{theorem}\label{thm:q2}
	If $G$ is a graph with $n$ vertices, $m<n-1$ edges, and minimum degree at least one (so that $m\geq n/2$), then $\mu(G)\leq \mu(K_{1,p}\cup qK_2)$ where $p=2m-n+1$ and $q=n-m-1$.
\end{theorem}

\begin{proof}
	We may suppose that $G$ is a graph satisfying these conditions maximizing $\mu$ and having the least number of components.  If $G$ has components $C_1,C_2,\ldots,C_k$, then $\mu(G)=\sum_{i=1}^k \mu(C_i)$.  Thus, by Theorem~\ref{thm:lex}, we may assume that each component is a lex graph.  Our first aim is to prove that we may suppose that every component is a star.  It cannot be that each component is a non-tree, for then $G$ would have at least $n$ edges.  Suppose, then, that $C_1$ is a non-tree component and $C_2$ is a tree component (and so therefore a star).  We have $e(C_1)+e(C_2)\geq n(C_1)+n(C_2)-1$, so we have enough edges to build a connected lex graph $L$ on $V(C_1)\cup V(C_2)$.  Thus, $\mu(L)\geq \mu(C_1)+\mu(C_2)$ and we have decreased the number of components, a contradiction.  
	
	Lastly, we will show that all but at most one component of $G$ is an edge.  Suppose $k,\ell\geq 2$ and $G$ contains components $K_{1,k}$ and $K_{1,\ell}$.  We can replace these with $K_{1,k+\ell-1}\cup K_2$ and strictly increase $\mu$, since
	\[
		\binom{k}2+\binom{\ell}2<\binom{k+\ell-1}2.
	\]
	This contradiction finishes the proof.
\end{proof}

% subsection mu (end)

% section folding (end)

\subsection{Discharging} % (fold)
\label{sub:discharging}

In the proof of the Main Theorem, one case occurs when the average weight of all edges of $G$ is at most $r-2$.  In this section, we give a technique to bound the weight of edges around a cluster.  The only edges having weight $r-1$ are, of course, in clusters.  It is convenient to let the \emph{benefit} of an edge be defined by $\wbar(xy)=r-2-\w(xy)$.  An edge joining a cluster $T$ to $S=S_T$ will have benefit $d_R(v)-1$.  Note that if $e$ is a tight edge, then $\wbar(e)=-1$ and all other edges have $\wbar(e)$ nonnegative.  

  Such an edge may be incident with two clusters, so we will use the benefit of these edges by transferring $\frac{1}2(d_R(v)-1)$ from such an edge to the edges in $T$.  This ``discharging'' will sometimes be enough to show that average weight of an edge in $G$ is at most $r-2$, i.e., that the average benefit is nonnegative.  This is sufficient to show that $k_3(G)\leq k_3(aK_{r+1}\cup K_b)$, as we prove in the final section.

\begin{definition}
	A cluster $T$ is \emph{dischargeable} if the associated graph $R$ satisfies
	\[
		\sum_{v\in S} (d_R(v)-1)\geq t-1,
	\]
	i.e., if $2e(R)\geq s+t-1$.
\end{definition}

\begin{lemma}\label{lem:discharge}
	If every cluster of $G$ is dischargeable, then the average weight of all edges is at most $r-2$.
\end{lemma}

\begin{proof}
	We partition the edges of $G$ into those in clusters, those associated with clusters (meaning they are incident to a vertex in a cluster, but not in a cluster), and others.  Clearly, any edge is associated with at most two clusters.  We transfer half the benefit of each edge associated to a cluster to that cluster, reducing the benefit of that edge and increasing the total benefit of the cluster edges.  The statement that a cluster is dischargeable is precisely that, after this process, the total benefit of the cluster edges is nonnegative.  Each edge incident to a cluster has transferred half of its benefit at most twice and hence still has nonnegative benefit.  Edges not incident to clusters all have nonnegative benefit.  Thus, the total benefit of all edges is nonnegative.
\end{proof}
	
% sub discharging (end)

\section{Proof of Main Theorem} % (fold)
\label{sec:theorem}

In this section, we first prove that every cluster is either dischargeable or foldable.

\begin{theorem}\label{thm:mainlem}
	If $T$ is a cluster in $G$, then $T$ is either dischargeable or foldable.
\end{theorem}

In order to prove this, we need a couple of lemmas.

\begin{lemma}\label{lem:starmatch}
	If $T$ is a cluster with $R=K_{1,p}\cup qK_2$, then $T$ is either dischargeable or foldable.
\end{lemma}

\begin{proof}
	Note that we have $s=p+1+2q$ and $e(R)=p+q$.  Also, $\mu(R)=\binom{p}2$.  Then $T$ is dischargeable if
	\[
		2e(R)=2(p+q)\geq s+t-1=p+2q+t.
	\]
	This holds when $t\leq p$.  If not, i.e., if $t>p$, then
	\[
		t\cdot e(R)=t(p+q)\geq \binom{p}2=\mu(R),
	\]
	and so $T$ is foldable.
\end{proof}

\begin{lemma}\label{lem:density}
	If $G$ is any graph, then $\mu(G)\leq \frac{2}3(n-2)e(G)$.  In particular, if $T$ is a cluster satisfying $t\geq \frac{2}3(s-2)$, then $T$ is foldable. 
\end{lemma}

\begin{proof}
	Let $G$ be a graph with vertex set $V$ and edge set $E$.  For $Q\in \binom{V}3$, we define $\mu(Q)$ by $\mu(\tri)=2$, $\mu(\cher)=1$, and $\mu(\angel)=\mu(\emp)=0$.  That is, $\mu(Q)$ is the contribution of $Q$ to $\mu(G)$.  Also, we let $e(Q)=e(G[Q])$ and $E(Q)=E(G[Q])$.  We compute as follows.
	\[
		\mu(G)=\sum_{Q\in \binom{V}3} \mu(Q)
		=\sum_{\substack{Q\in \binom{V}3\\e(Q)\neq 0}}\frac{1}{e(Q)} \sum_{e\in E(Q)} \mu(Q)
		=\sum_{e\in E} \sum_{Q: e\of Q} \frac{\mu(Q)}{e(Q)}
		\leq \frac{2}3(n-2)e(G).
	\]	
The inequality comes from the fact that $\mu(Q)/e(Q)\leq 2/3$ for every triple $Q$ containing edges.  

Now, for the second claim, we have $\mu(R)\leq \frac{2}3 (s-2)e(R)$.  Therefore, by hypothesis,
\[
	t\cdot e(R)\geq \frac{2}3 (s-2)e(R)\geq \mu(R),
\]
and so $T$ is foldable.
\end{proof}

We are now ready to prove our main lemma, Theorem~\ref{thm:mainlem}.

\begin{proof}[Proof of Theorem~\ref{thm:mainlem}]
	By Lemma~\ref{lem:density}, we are done if $t\geq \frac{2}3(s-2)$.  Also, $T$ is dischargeable if $2e(R)\geq s+t-1$.  Thus, in the remaining cases, we have
	\[
		e(R)<\frac{s+t-1}2<\frac{5}6s-\frac{7}6.
	\]
	Let $m=e(R)$.  We know that whether $T$ is foldable depends only on $t$, $m$ and $\mu(R)$, whereas whether $T$ is dischargeable depends only on $t$ and $m$.  By Theorem~\ref{thm:q2}, since $\d(R)\geq 1$, we have that $\mu(R)\leq \mu(K_{1,p}\cup qK_2)$ where $q=s-m-1$ and $p=2m-s+1$.  Thus, if $T$ would be foldable with $R=K_{1,p}\cup qK_2$, then $T$ is foldable.  Therefore, $T$ is foldable (dischargeable) if $T$ is foldable (dischargeable) with $R=K_{1,p}\cup qK_2$.  By Lemma~\ref{lem:starmatch}, we are done.
\end{proof}

It only remains to show that, provided $a\geq 2r/(3\sqrt{3})$, $G$ must contain a cluster, i.e., $G$ must contain an edge of weight $r-1$. 

\begin{lemma}\label{lem:avgwt}
	If $a\geq 2r/(3\sqrt{3})$, then 
	\[
		\sum_{e\in E(aK_{r+1}\cup K_b)} w(e)> (r-2)\frac{r(a(r+1)+b)}{2},
	\]
	i.e., any graph of maximum degree at most $r$ on $a(r+1)+b$ vertices either has fewer triangles than $aK_{r+1}\cup K_b$ or has an edge of weight $r-1$.
\end{lemma}

\begin{proof}
	The proof is a straightforward calculation.  
\end{proof}

We are now ready to prove the main result of the paper.

\begin{proof}[Proof of Main Theorem]
	It suffices to prove that, if $G$ is a graph with $\D(G)\leq r$ having the greatest number of triangles, then, provided $a\geq 2r/(3\sqrt{3})$, $G$ contains a copy of $K_{r+1}$.  By Lemma~\ref{lem:avgwt}, the average weight of edges in $G$ is such that some edge must have weight $r-1$.  Thus, $G$ must have a cluster, and by Theorem~\ref{thm:mainlem}, all clusters are foldable or dischargeable.  If all clusters were dischargeable, the average weight would be at most $r-2$, contradicting Lemma~\ref{lem:avgwt}.  If $G$ has a foldable cluster $T$, then $k_3(G)\leq k_3(G_T)$ and $G_T$ contains a $K_{r+1}$ so we are done by induction. 
	
	We now consider the case when $r\leq 6$.  We apply our techniques to the graph $H$ guaranteed by the above, so let $n=n(H)$ and redefine $a$ and $b$ as usual.  It suffices to show that if $r\leq 6$ and $H$ is a graph on at most $2r(r+1)/(3\sqrt{3})+r$ vertices having average weight of edges at most $r-2$, then $k_3(H)\leq k_3(aK_{r+1}\cup K_b)$.  I.e., we need to verify 
	\[
		(r-2)\frac{r(a(r+1)+b)}{6}\leq a\binom{r+1}3+\binom{b}3.
	\]
	This is a finite (and simple) calculation since $a$ is bounded above.  See \cite{Calcs} for details.
\end{proof}

% section proof_of_main_theorem (end)

\bibliographystyle{amsplain}
\bibliography{maxdegk3}

\end{document}